\documentclass[12pt]{amsart}
\usepackage{amsmath, amssymb,mathabx}

\usepackage{mathrsfs}
\usepackage{stmaryrd}
\usepackage{xypic ,color}
\usepackage[backend=biber, style=alphabetic, maxnames=10,
giveninits=true, sorting=nyt]{biblatex}
\usepackage{url}
\usepackage{tikz-cd}
\usepackage{varwidth}

\usepackage[colorlinks=true,linkcolor=blue,urlcolor=blue,citecolor=blue]{hyperref}

\theoremstyle{definition}

\numberwithin{equation}{section}

\newcommand\Spec{\operatorname{Spec}}

\newcommand\pone{\mathbb{P}^1}

\newcommand{\op}{\operatorname}

\newcommand\kz{\mathcal{KZ}_{\kappa}}
\newcommand\tensor{\otimes}
\newcommand\ml{\mathcal{L}}
\newcommand\starr{\bullet}

\newcommand\dlog{\operatorname{dlog}}

\newcommand\mc{\mathcal{C}}

\newcommand\rk{\operatorname{rk}}

\newcommand\frg{\mathfrak{g}}
\newcommand\frh{\mathfrak{h}}
\newcommand\cblock{\mathcal{V}_{\ell}}

\newcommand{\nc}{\newcommand}

\nc{\rnc}{\renewcommand}

\nc{\vv}{\widearrow}

\renewcommand{\l}{\lambda}
\nc{\ov}{\overline}

\nc{\vl}{\vv{\lambda}}
\nc{\vgnl}{\mathbb{V}_{g, n , \vl, \ell}}

\nc{\Q}{\mathbb{Q}}
\nc{\Z}{\mathbb{Z}}
\nc{\C}{\mathbb{C}}

\newcommand{\cb}{\color{blue}} %
\newcommand{\cora}{\color{orange}}

\newcommand{\Najmuddin}[1]{\framebox{\begin{varwidth}{0.9\textwidth}
      Najmuddin: #1 \end{varwidth}}\newline}
\newcommand{\Prakash}[1]{\framebox{\begin{varwidth}{0.9\textwidth}
      Prakash: #1 \end{varwidth}}\newline}

\rnc{\sl}{\mathfrak{sl}}

\newcommand\SL{\operatorname{SL}}

\newcommand{\leto}[1]{\stackrel{#1}{\to}}

\newtheorem{theorem}{Theorem}[section]
\newtheorem{remark}[theorem]{Remark}

\newtheorem{corollary}[theorem]{Corollary}
\newtheorem{question}[theorem]{Question}
\newtheorem{proposition}[theorem]{Proposition}

\newtheorem{lemma}[theorem]{Lemma}
\newtheorem{clm}[theorem]{Claim}

\newtheorem{defi}[theorem]{Definition}

\newtheorem{example}[theorem]{\bf Example}

\begin{document}

\title[Conformal blocks and the KZ connection]{Conformal blocks in genus zero and the KZ connection}
\author{Prakash Belkale and Najmuddin
  Fakhruddin}

\begin{abstract} We survey some recent work on conformal blocks in
  genus zero, focussing on (1) Chern classes, global generation and
  morphisms, and (2) the Knizhnik--Zamolodchikov connection on
  conformal blocks (and invariants), their motivic realisations, and
  unitarity.

\end{abstract}
\maketitle

\vspace{-.5cm}
\begin{flushright}
  {\it To the memory of Alberto Collino}
  \end{flushright}
\vspace{.5cm}
\section{Introduction}

Let $\frg$ be a finite dimensional simple Lie algebra over the field
of complex numbers, $\frh\subset \frg$ a Cartan subalgebra.  Fix a
Cartan decomposition of $\frg$. Let $R=R^+\cup R^- \subseteq \frh^*$
be the set of roots and $\Delta\subset R$ be the set of simple
positive roots. Let $\theta\in R^+$ be the highest root. Let $(\ ,\ )$
be the Killing form normalised so that $(\theta,\theta)=2$ and let
$\Omega$ denote the corresponding (normalised) Casimir element. For a
dominant integral weight $\lambda\in\frh^*$, let $V_{\lambda}$ denote
the corresponding irreducible representation of $\frg$.  A
  dominant weight $\lambda$ is said to be of level
  $\ell\in \Bbb{Z}_{\ge 0}$ if $(\lambda,\theta)\leq \ell$.

 Let { $g,n$ be non-negative integers with $g +n -3\geq 0$} and let
 $\vv{\lambda}=(\lambda_1,\dots,\lambda_n)$ be an $n$-tuple of
  dominant integral weights of $\frg$ of some level $\ell$. The
theory of conformal blocks produces a vector bundle
$\mathbb{V}_{g,n,\vv{\lambda},\ell}$ of conformal blocks on
$\overline{M}_{g,n}$, the moduli stack of { stable} genus $g$ curves with $n$
marked points. The restriction of this vector bundle to ${M}_{g,n}$
carries a projective integrable connection. We will drop $g$ or $\ell$
from the notation if they are clear from the context.

In genus $0$, $\mathbb{V}_{n,\vv{\lambda},\ell}$ is surjected on by the
constant vector bundle with fibres the
coinvariants $$\mathbb{A}(\vv{\lambda})=V(\vv{\lambda})_{\frg}=V(\vv{\lambda})/\frg
V(\vv{\lambda}), $$
where  $V(\vv{\lambda})=V_{\lambda_1}\tensor V_{\lambda_2}\tensor\cdots\tensor V_{\lambda_n}$.  Working over the configuration space of $n$ points on the affine line, the projective connection on  $\mathbb{V}_{n,\vv{\lambda},\ell}$ lifts to the KZ  connection on the space of coinvariants  $\mathbb{A}(\vv{\lambda})$.

We survey some of the recent work and pose questions---some of which
are well-known---on the following:
\begin{enumerate}
\item Chern classes of the conformal block bundles in genus $0$, with
  focus on the positivity properties.
\item The monodromy and geometric realization of the KZ connection on
  coinvariants, as well as on conformal blocks (mostly) in genus $0$.
\end{enumerate}

The structure of this article is as follows:

In \S \ref{s:basic} we recall some of the basic properties of
conformal block bundles. Since there exist many surveys on this
material we restrict ourselves to a few statements (with references)
and pose a question about the integral properties of these bundles.

Sections \ref{s:chern} and \ref{s:KZ} can be read independently of
each other. In \S \ref{ss:chern} we survey some of what is known about
the Chern classes of conformal blocks bundles on $\ov{M}_{0,n}$. The
main interest in this is due to the fact that these bundles are
generated by their global sections, hence give rise to morphisms from
$\ov{M}_{0,n}$ to Grassmannians. A detailed understanding of the class
of all morphisms that arise in this way is still missing, but we
discuss various connections to morphisms arising from GIT and other
classical constructions. The best understood case is when
$\frg = \sl_m$ and $\ell = 1$ in which case there is also a connection
to Gromov--Witten theory. While the first Chern classes of these
bundles have been explored to some extent, very little is known about
the higher Chern classes and the higher codimension classes obtained
by performing various algebraic operations on them. In particular, one
may construct basepoint free divisor classes by pushing forward higher
codimension classes constructed in this way from $\ov{M}_{0,m}$, for
$m >n$, via the natural degeneracy maps. Some constructions and
questions about these topics are considered in \S \ref{ss:pos} and \S
\ref{ss:sym}.

In \S \ref{s:KZ} we survey some properties of the KZ connection on
conformal blocks and invariants, focussing mainly on the
geometrization of this connection. This part relies on the
work of Looijenga, Ramadas, Schechtman and Varchenko. The main theme
is that the connection is of Gauss--Manin type, so there is an
underlying motivic sheaf and, in particular, a variation of mixed
Hodge structure. We discuss several open questions about the numerical
invariants of these structures as well as their behaviour under
natural duality and symmetry operations on conformal blocks. We note
that this section, though mostly expository, contains some
improvements and extensions of known results.

\smallskip

We have not attempted to be exhaustive in this survey and have left
out some topics of active research in related areas. These include the
theory of twisted conformal blocks \cite{DC,HK} and a formula for
their dimensions \cite{DM}, as well as sheaves of conformal blocks on
$\overline{M}_{g,n}$ for vertex operator algebras (see the recent survey
\cite{DGT}).

\smallskip

N.F. acknowledges the support of DAE, Government of India, under
Project Identification No. RTI 4001.

\section{Basic properties of conformal block bundles} \label{s:basic}

The bundles of conformal blocks were first defined by Tsuchiya, Yamada
and Ueno in \cite{TUY} (where they were called sheaves of covacua) for
all $g$ using the representation theory of the affine Lie algebra
associated to $\frg$. Their definition used local coordinates, so did
not immediately give rise to bundles on the stacks $\ov{M}_{g,n}$;
coordinate free definitions were later given in \cite{tsuchi},
\cite{sorger} and \cite{loo-wzw} using which one can define the
bundles $\mathbb{V}_{g,n,\vv{\lambda},\ell}$ on $\ov{M}_{g,n}$; see
\cite[\S 2.2]{Fakh} for details. We will not use the explicit
definition of these bundles, so we do not explain it here, but we
recall some of their fundamental properties.  For any dominant weight
$\l$ we denote by $\l^*$ the weight corresponding to the dual of
$V_{\l}$.

\begin{proposition} \label{p:compat} 
$ $
\begin{enumerate}
\item Let $f_i: \ov{M}_{g,n} \to \ov{M}_{g,n-1}$ be the
  morphism given by forgetting the $i$-th marked point.  If $\l_i = 0$
  for some $i$, then $\vgnl \cong f_i^*
  (\mathbb{V}_{g,n,\vl', \ell})$, where $\vl'$ is obtained from $\vl$
  by deleting $\l_i$.
\item Let $ \gamma: \ov{M}_{g_1,n_1+1} \times
  \ov{M}_{g_2,n_2+1} \to \ov{M}_{g_1 + g_2,n_1+n_2}$ be the
  gluing morphism, where we glue along the last marked point for each
  factor. Then  we have an
  isomorphism
  \[
  \gamma^*(\mathbb{V}_{g_1 + g_2,n_1+n_2,\vl, \ell}) \stackrel{\sim}{\to}
  \bigoplus_{\mu} \mathbb{V}_{g_1,n_1+1,\vv{\lambda^1 \mu},\ell} \otimes
  \mathbb{V}_{g_2,n_2 + 1,\vv{\lambda^2\mu^*}, \ell}
  \]
  where $\vv{\l^1\mu} := (\l_1,\dots,\l_{n_1},\mu)$ and $\vv{\l^2  \mu^*}
  := (\l_{n_1+1},\dots,\l_{n_1+n_2},\mu^*)$.
\item Suppose $g > 0$. Let $\gamma:\ov{M}_{g-1,n+2} \to \ov{M}_{g,n}$
  be the gluing morphism where we glue the last two marked points to
  each other. Then we have an isomorphism
  \[
  \gamma^*(\vgnl) \stackrel{\sim}{\to} \bigoplus_{\mu}
  \mathbb{V}_{g-1,n+2, \vv{\l\mu}, \ell}
  \]
  where $\vv{\l\mu} := (\l_1,\dots,\l_n,\mu,\mu^*)$.
\end{enumerate}
(The sums in (2) and (3) are over  all dominant integral weights $\mu$ of level $\ell$ of $\frg$.)
\end{proposition}
The proposition reduces the computation of the ranks of the bundles
$\vgnl$ to the case $g=0$ and $n=3$. The Verlinde formula, see
\cite[Section 4]{sorger}, gives an explicit formula for these ranks.

\begin{remark} \label{r:cft}
 An immediate consequence of Proposition \ref{p:compat} is that the
 Chern characters of the bundles $\vgnl$ give rise to a cohomological
 field theory in the sense of \cite[\S 6]{koma}.
\end{remark}

Another basic property of the $\vgnl$ is the existence of the WZW
connection, which was first constructed mathematically in \cite{TUY};
see also \cite{ueno}, \cite{tsuchi}, \cite{sorger} and \cite{loo-wzw}.
\begin{proposition} \label{p:wzw}
  The restriction of $\vgnl$ to $M_{g,n}$ has a natural flat
  projective connection with logarithmic singularities along the boundary.
\end{proposition}
The connection is quite explicit, and allows one to compute the
monodromy around each boundary component (which is always semisimple
and of finite order). When $g=0$ this projective connection lifts to a
flat connection on the conformal block bundles called the
Knizhnik--Zamolodchikov (henceforth KZ) connection; we discuss this in
more detail in \S \ref{s:KZ}.


\begin{question} \label{q:Z}
  The moduli stacks $\ov{M}_{g,n}$ are defined over $\Spec{(\Z)}$ and using the
  split form of $\frg$ defined over $\Q$ one can define all the
  conformal blocks bundles (together with the connections) over $\Spec{(\Q)}$.
  Are there natural models of these bundles defined over $\Spec(\Z)$?
\end{question}

\begin{remark} \label{r:rank}%
  Conformal blocks bundles can be geometrically described in terms of
  spaces of sections of line bundles over suitable stacks of parabolic
  principal bundles on curves, see e.g., \cite{KBook}. Although we
  will not use it explicitly in this paper, this geometric viewpoint
  is useful in proving properties of conformal block bundles.

  These line bundles can be defined in any characteristic and Question
  \ref{q:Z} is closely related to the question whether their spaces of
  global sections have constant rank over all points of
  $\ov{M}_{g,n/\mathrm{Spec}(\Z)}$. See, e.g., \cite{faltings-lg} for
  some relevant constructions and results.
\end{remark}

\section{Chern classes and positivity in genus $0$} \label{s:chern}

\subsection{Chern classes} \label{ss:chern}

The Verlinde formula computes the ranks for the conformal block
bundles $\mathbb{V}_{g,n,\vv{\lambda},\ell}$. To get information about
more global properties of these bundles it is natural to try to
compute the Chern classes of these bundles. This was carried out in
\cite{Fakh} when $g=0$ and $c_1$ was (implicitly) computed for all $g$
there;\footnote{A formula for $c_1$ of the restriction of these
  bundles to $M_{0,n}$ also follows from \cite{tsuchi}.} an explicit
formula for the Chern character for all $g$ was later given in
\cite[Theorem 1]{MOPP}. The method of \cite{Fakh} for $c_1$ depends
only on the residue theorem (for curves) and computation of the
residues of the WZW connection along the boundary divisors, whereas
the method of \cite{MOPP} uses the classification of semi-simple
cohomological field theories from \cite{teleman} (cf. Remark
\ref{r:cft}).

We do not reproduce these general formulas here but to give the reader
the flavour we write down the formula in the case $g = 0$ and $n=4$,
in which case $\ov{M}_{0,4} \cong \pone$. To simplify notation we
denote the rank of the bundle $\mathbb{V}_{0, n, \vv{\l},\ell}$ by
$r_{\vv{\l}}$; the integer $n$ is determined by $\vv{\l}$ and $\ell$
is assumed to be fixed.  For any dominant weight $\l$ we denote
by $c(\l)=(\l,\l+2\rho)$ the value of the Casimir element $\Omega$ acting on
$V_{\l}$, where $\rho$ is half the  sum of positive roots of $\frg$. Then we have the following formula (\cite[Corollary
3.5]{Fakh}):
\begin{proposition} \label{p:4points}
  Let $\frg$ be a simple Lie algebra, $\ell$ a non-negative integer and
  $\vv{\l}$ a $4$-tuple of dominant weights for $\frg$ of level $\ell$.
  Then
    \begin{multline}
\label{eq:mainformula}
\deg(\mathbb{V}_{4, \vv{\l},\ell}) = \frac{1}{2(\ell + h^{\vee})} \times
\biggl\{ \Bigr \{ r_{\vv{\l}} \sum_{i=1}^4 c(\l_i)
\Bigr \} \ -  \\
\Bigl \{ \sum_{\l} c(\l)\bigl \{ r_{(\l_1, \l_2, \l)}
\cdot r_{(\l_3, \l_4, \l^*)} + r_{(\l_1, \l_3, \l)}\cdot r_{(\l_2,
  \l_4, \l^*)} + r_{(\l_1, \l_4, \l)}\cdot r_{(\l_2, \l_3, \l^*)}
\bigr \}
\Bigr \} \biggr \} \ .\\
\end{multline}
 Here $\lambda$ runs through all dominant integral weights of level $\ell$ of $\frg$.
\end{proposition}

Proposition \ref{p:compat} and the fact that $CH_1(\ov{M}_{0,n})$ is
generated by the classes of its $1$-dimensional strata (see
\cite{keel}) implies that \eqref{eq:mainformula}
determines the first Chern class when $g=0$; for general $g$ one also
needs a formula for the degree of conformal block bundles on
$\ov{M}_{1,1}$ (\cite[Theorem 6.1]{Fakh}).

\begin{example} \label{ex:sl2} Let $\frg = \sl_2$, let  $\ell= 1$ and
  let each $\l_i$, $i=1,2,3,4$, be the fundamental weight. In this
  case the formula \eqref{eq:mainformula} implies that the degree is
  equal to $1$.
\end{example}

In \cite[Lemma 2.5]{Fakh} it was observed that when $g=0$ all the
bundles $\mathbb{V}_{g,n,\vv{\lambda},\ell}$ are generated by their
global sections. Therefore, these bundles give rise to morphisms from
$\overline{M}_{0,n}$ to projective spaces (using their first Chern
class) or, in general, to Grassmannians. The fact that these morphisms
are interesting, e.g., not just constant, follows from the
computations of the Chern classes (see Example \ref{ex:sl2}).  On the
other hand, for $g>0$ it was observed in \cite[\S 6.1]{Fakh} that the
conformal block bundles are usually not even nef and so are perhaps
not very useful as a tool for investigating the geometry of
$\overline{M}_{g,n}$.


We also remark that for $g=0$ and fixed $n$ and $\vv{\lambda}$, the
bundles are trivial for large $\ell$ and so the positivity properties
of these bundles range from being ample to being trivial, but in
general it is not known when exactly either of these cases occur (even
if $n=4$).

\begin{remark}
It has been shown in \cite{DamGib} that global generation may fail even in genus zero for the wider class of conformal blocks given by (appropriate)  VOAs.
\end{remark}

Two special cases have been studied in some detail: The first Chern
classes of level one conformal block bundles for $\sl_n$ and the first
Chern classes of conformal block bundles for $\sl_2$
\cite{Fakh,AGS,AGSS,BoGi,GG,GJMS}. The numerical computations were
carried out in many of these papers using Swinarski's software package
\cite{S}. The first Chern classes of level one conformal block bundles
on $\sl_n$ arise in many different geometric settings (see, e.g.,
\cite[\S 3]{fed-new}) and they can be
considered the best studied class of conformal block divisors.
On the other hand, a large class of basepoint free divisors
on $\ov{M}_{0,n}$ is constructed in \cite{BG2} using the Gromov--Witten
theory of projective homogenous spaces, and in the special case of
projective space these classes are identified with level one conformal
block divisor for $\sl_n$, \cite[Theorem 3.1]{BG2}.

We now { briefly} describe some of the known results about the
level $1$ classes, mostly for the $\sl_m$, and also some results for
$\sl_2$ classes at higher levels.

\subsubsection{}

Level one conformal block bundles for $\sl_2$ give a natural,
$S_n$-equivariant basis for the Picard group of $\ov{M}_{0,n}$
consisting of basepoint free line bundles (\cite[Theorem
4.3]{Fakh}). The basic case is when $n$ is even and $\vl$ is given by
$\lambda_i = \varpi$, the fundamental weight, for all $i$. The first
Chern classes of the bundles with the same $\vl$, but for all $\ell$,
have been studied in \cite{AGS}; in particular, for
$\ell = 1,2,g-1,g$, where $n = 2g+2$, they show that the morphisms
corresponding to these divisors have natural geometric
descriptions. For example, when $\ell = 1$ the line bundle is pulled back
from the Satake compactification $A_g^*$ of $A_g$ (the moduli stack of
principally polarised abelian varieties of dimension $g$) via the
morphism given by taking the double cover of the genus $0$ curve
ramified along the marked points and then composing with the Torelli
map (\cite[Theorem 6.2]{AGS}).

\subsubsection{}

Continuing with $\frg = \sl_2$, but now letting $\vl$ be arbitrary
with all $\l_i \neq 0$, we define the associated critical level to be
$(\sum_i \l_i/2)-1 $   using the identification of dominant integral weights with $\Bbb{Z}_{\ge 0}$, with $\varpi$ corresponding to $1$. In this case, it was shown in \cite[Theorem
4.5]{Fakh} that the determinants of the conformal bundles are pulled
back from ample bundles on GIT quotients $(\pone)^n \sslash \SL_2$,
with the polarisation determined by $\vl$. For general $\ell$, the
bundles are always pulled back from a map
$\ov{M}_{0,n} \to \ov{M}_{0, \mathcal{A_{\vl}}}$, where the target is
one of Hassett's moduli spaces of weighted pointed stable curves of
genus $0$ (see \cite[Proposition 4.7]{Fakh}), however the bundle on
the Hassett space need not be ample; the precise morphism defined by
the base point free linear system is unknown in general.

\subsubsection{} \label{s:cone}

A natural question is to ask for the structure of the subcone of the
nef cone of $\ov{M}_{0,n}$ generated by the first Chern classes of
conformal blocks bundles on $\ov{M}_{0,n}$.  Not much seems to be
known about this question, even for a fixed $\frg$, but it is shown in
\cite[Theorem 1.1]{GG} that the cones generated by level $1$ divisors
for $\sl_m$, with $m$ varying, are finitely generated. The proof shows
that these cones are spanned by finitely many cones arising from GIT
constructions, each of which is known to be finitely
generated. However, it is known that the subcone of the nef cone
spanned by all conformal block divisors (even just those of type $A$)
is strictly larger than this cone: in \cite{swin-not} it is shown that
for $n=8$ there is a level $2$ conformal block divisor for $\sl_2$
which is not contained in the level $1$ cone.

\subsubsection{}

Further connections with GIT are given in, among others, \cite{GJMS}
and \cite{BoGi}. In \cite{GJMS}, certain conformal block divisors for
$\sl_m$ are related to GIT quotients of spaces of pointed rational
normal curves, the so-called Veronese quotients constructed in
\cite{GJM}. In \cite[Theorem 1.7]{BoGi}, new proofs of results from
\cite{GG} and \cite{Fedo} are given by axiomatising the factorisation
formula for conformal block bundles (Proposition \ref{p:compat}) and
showing that certain line bundles arising from GIT have a similar
property \cite[Theorem 1.2]{BoGi}.

\subsubsection{}

For general $\frg$, some basic properties of the bundles with the
$\l_i$ of level $1$, but $\ell$ arbitrary were established in \cite[\S
5.1]{Fakh}. For type $D$, it can be seen by using \cite[Proposition
5.6]{Fakh} that the \emph{closure} of the cone spanned by all the
level $1$ divisors (allowing the rank of the Lie algebra to vary) is
the same as the cone spanned by $\sl_2$ level $1$ divisors and the
same is true for $E_7$. For $E_6$ the cone is the same as for $\sl_3$
at level $1$.  The most interesting case at level $1$, aside from type
$A$, is perhaps type $C$ (since all fundamental weights are of level
$1$, as in type $A$).

All the available evidence suggests that the first Chern class of
conformal block bundles for arbitrary $\frg$ is contained in the
subcone of the nef cone generated by the first Chern class of type $A$
bundles, though a conceptual reason for this seems to be missing. In
any case, we ask:
\begin{question} \label{q:allg}%
  Is the first Chern class of all $\mathbb{V}_{n,\vv{\lambda},\ell}$
  contained in the subcone of the nef cone of $\ov{M}_{0,n}$ spanned
  by all the first Chern classes of type $A$ conformal block bundles
  on $\ov{M}_{0,n}$?
\end{question}

\subsection{Positivity} \label{ss:pos}

To formulate the positivity results and questions it is useful to
introduce the notion of strongly base point free cycles following
\cite{FL}. We use the following variant \cite{BG2}:
\begin{defi} %
  An effective integral Chow cycle $\alpha \in A_k(X) $ of codimension
  $k$ on an equidimensional projective variety $X$ is said to be
  rationally strongly basepoint free if there is a flat morphism
  $s: U \to X$ from an equidimensional quasi-projective scheme $U$ of
  finite type and a proper morphism $p : U\to W$ of relative dimension
  $\dim X - k$, where W is isomorphic to an open subset of $\mathbb{A}^m$
  for some $m$, such that each component of $U$ surjects onto $W$, and
  $\alpha = (s|_{F_p})_*[F_p]$, where $F_p$ is a general fibre of $p$.

Denote the semigroup of rationally strongly basepoint free classes of
codimension $k$ on a 
projective variety $X$ by $SBPF_k(X) \subseteq A_k(X).$
\end{defi}

Strongly basepoint free classes are base point free: for every cycle
$V$, they can be represented by a cycle intersecting $V$ in no more
than the expected dimension. They are compatible with various
operations in intersection theory  \cite{FL} and \cite[Lemma 1.3]{BG2}
like intersections, flat pushforward and pullback. A divisor class on
a smooth projective variety is base point free if and only if it is strongly base point free.

Schur polynomials of globally generated vector bundles are
(rationally) strongly base free. Applying this to the globally
generated vector bundle of conformal blocks
$\mathbb{V}=\mathbb{V}_{n,\vv{\lambda},\ell}$ on $\overline{M}_{0,n}$, we
obtain the following: Let $e=\rk \mathbb{V}$, and let
$\vv{\mu} = (\mu_1 \geq\dots \geq \mu_m \geq 0)$ be a partition with
$\mu_1\leq e$. Then, the Schur polynomial
$s_{\mu} = \det(c_{\mu_i+j-i})_{1\leq i,j\leq m}$ in the Chern classes
$c_i = c_i(\mathbb{V})$ is (rationally) a strongly base point free class
of codimension $|\mu|=\sum \mu_i$. This is because these Schur classes
are pullbacks of Schubert varieties via the map from
$\overline{M}_{0,n}$ to the Grassmannian (which is homogenous)
corresponding to the globally generated $\mathbb{V}$.

Basepoint free divisor classes  on $\overline{M}_{0,n}$ determine morphisms from $\overline{M}_{0,n}$ to projective spaces. It is of interest to classify the images of such morphisms and also study the extremality property of the corresponding element of the nef cone of $\overline{M}_{0,n}$.

The first Chern classes of
$\mathbb{V}_{n,\vv{\lambda},\ell}$ on $\overline{M}_{0,n}$
give such divisor classes, but we also get a collection of classes
from the Schur classes $s_{\mu}$ as above on
$\overline{M}_{n-|\mu|+1}$ if we push down the class $s_{\mu}$ by a
marked point forgetting flat map
$\overline{M}_{0,n}\to \overline{M}_{0,n-|\mu|+1}$. Similarly, we may
get nef divisor classes by intersecting two (or more) of these classes
(for possibly different conformal block bundles) on $\ov{M}_{0,m}$,
for some $m > n$, and pushing down via marked point forgetting maps.

\begin{question} \label{q:expand} %
  The (nonzero) level one conformal block bundles for $\sl_m$ are of
  rank one.  As discussed in \S \ref{s:cone}, the subcone of the
    nef cone generated by all level one conformal block bundles is
    finitely generated.  Does one obtain a larger subcone of the nef
  cone on $\overline{M}_{0,n}$ by intersecting Chern classes of
  several level one conformal block bundles on $\ov{M}_{0,m}$ for some
  $m > n$ and then pushing down to $\ov{M}_{0,n}$?

  Similarly, do higher Chern classes and algebraic operations produce
  a bigger cone of nef divisors for $\sl_2$ than the one generated by
  the first Chern classes?
\end{question}

See \cite{swin-sl2} for some computational work on conformal block
divisors. It would be very interesting to have computational data related to
higher codimension classes, especially in relation to Question
\ref{q:expand}.

\subsubsection{}

The variety $M_{0,n}$ can be viewed as the complement of a hyperplane
arrangement in $\mathbb{A}^{n-3}$ by fixing three of the marked points to
be $0$, $1$ and $\infty$. A consequence of the Quillen--Suslin theorem
is that all vector bundles on such a variety are trivial.\footnote{See
https://mathoverflow.net/questions/426957/triviality-of-vector-bundles-on-affine-open-subsets-of-affine-space.}
However, as we have seen, conformal blocks bundles are not in general
  trivial on $\ov{M}_{0,n}$.

  There are some sufficient conditions that guarantee that the first
  Chern class of a conformal block bundle is trivial, see
  \cite{Fakh,BGM1,BGM2}: all these follow from criteria under which
  conformal blocks coincide with coinvariants.
\begin{question}
  Can one find necessary and sufficient
  conditions for the non-vanishing of the first Chern class of a
  conformal block bundle?
\end{question}
Note that the first Chern class is the pullback of an ample bundle on
a Grassmannian, hence if the first Chern class vanishes on
$\overline{M}_{0,n}$ then the map to the Grassmannian is trivial, and
the conformal blocks bundle is actually trivial. This question has
been answered for $\sl_2$ \cite{BGM2}. There are examples where conformal blocks are nonzero and do not coincide with coinvariants, and yet their first Chern classes are zero \cite[Section 4.2]{BGM2}.

\begin{remark} \label{r:F} %
  One of the main motivations for studying basepoint free and nef
  classes on $\ov{M}_{0,n}$ is the $F$-conjecture which gives a very
  natural conjectural description of the nef cone of $\ov{M}_{g,n}$,
  see \cite{GiKM} for a discussion. While the results discussed in
  this section provide a large family of basepoint free classes, the
  $F$-conjecture is still open. The related conjecture that
  $\ov{M}_{0,n}$ is a Mori dream space is now known to be false
  \cite{cas-tev} but, as far as we are aware, it is not known whether
  the nef cone is generated by finitely many semi-ample classes. We
  also note that \cite{pixton} gives an example of a basepoint free
  divisor which is not an effective linear combination of boundary
  divisors and a family of nef divisors which are not known to be
  basepoint free is constructed in \cite[\S 4]{fed-new}.
\end{remark}

\subsection{Symmetries} \label{ss:sym}%
The Chern classes of the bundles $\mathbb{V}_{n,\vv{\lambda},\ell}$
have some symmetry properties: These symmetries involve {\cb a change
  of} Lie algebras and levels.
\begin{enumerate}
\item Symmetries from strange dualities \cite{Mukh}.
\item Critical level symmetry: Recall that dominant integral weights
  of $\op{sl}_{r+1}$ are in one-one correspondence with ``normalised''
  Young diagrams
  $\lambda=(\lambda^{(1)}\geq
  \lambda^{(2)}\geq\dots\geq\lambda^{(r)}\geq \lambda^{(r+1)})$ with
  the normalisation $\lambda^{(r+1)}=0$.  Suppose $\vv{\lambda}$ is an
  $n$-tuple of normalised dominant integral weights for the Lie
  algebra $\op{sl}_{r+1}$.  Define the ``critical level'' $\ell_0$ by
  the equality $\sum{|\lambda_i|}=(r+1)(\ell_0+1)$. Then the first
  Chern classes of $\mathbb{V}_{n,\vv{\lambda},\ell}$ vanish for
  $\ell>\ell_0$ and for $\ell=\ell_0$ we have an equality of the first
  Chern classes of the vector bundles
  $\mathbb{V}_{n,\vv{\lambda},\ell}$ for the Lie algebra
  $\op{sl}_{r+1}$ and $\mathbb{V}_{n,\vv{\lambda^T},r}$ for the Lie
  algebra $\op{sl}_{\ell+1}$. These properties conjectured by Gibney
  were proved in \cite{BGM1}, using the Gepner--Witten equality
  relating the ranks of conformal blocks in genus zero with structure
  constants in the small quantum cohomology of Grassmannians.
    These results in the case $r=1$ are due to \cite{Fakh}.

\end{enumerate}
\begin{question}
Are there more examples of such symmetries, perhaps involving Schur polynomials and higher codimension basepoint free classes?
\end{question}

\section{Monodromy of the KZ connection and KZ motives} \label{s:KZ}
We first consider the KZ connection in its most general form (see
\cite[Section 1.4]{BaKi}, and \cite{L1}). Let $\mathcal{C}_n$ be the
configuration space of $n$ distinct points $(z_1,\dots,z_n)$ on the
complex affine line $\mathbb{A}^1$. There is a KZ connection on the
constant bundle on $\mathcal{C}_n$ with fibre
$V(\vv{\lambda})=V_{\lambda_1}\tensor
V_{\lambda_2}\tensor\cdots\tensor V_{\lambda_n}$. The connection
equations take the form: (Here
$f\in V(\vv{\lambda})\otimes\mathcal{O}_{\mathcal{C}_n}$)
$$\kappa\frac{\partial}{\partial z_i}f =\biggl(\sum_{j\neq
  i}\frac{\Omega_{ij}}{z_i-z_j}\biggr)f .$$ Here $\Omega_{ij}$ is the
Casimir element acting on the $i$ and $j$ tensor factors, and
$\kappa\in\mathbb{C}^*$.

The KZ connection is flat and commutes with the diagonal action of
$\frg$, hence it induces a connection on the constant bundle of
coinvariants
$\mathbb{A}(\vv{\lambda})=V(\vv{\lambda})_{\frg}=V(\vv{\lambda})/\frg
V(\vv{\lambda})$.  It also induces a connection on the various weight
spaces $V(\vv{\lambda})_{\nu}$, for the diagonal action of $\frh$ on
$V(\vv{\lambda})$. Here $\nu\in \frh^*$.
\begin{remark} %
  The monodromy of the KZ connection gives rise to a representation of
  the fundamental group of $\mathcal{C}_n$, the pure braid group on
  $n$ strands. There is an extensive theory of these representations
  (the Drinfeld--Kohno theorem, see, e.g., \cite[Theorem
  XIX.4.1]{Kassel}) when $\kappa$ is generic or irrational, with
  connections to the theory of quantum groups. Here we are mostly
  interested in the case when $\kappa$ is an integer when these
  general results do not directly apply.
\end{remark}

Following \cite{L1}, write
$$V_{\lambda_1}\tensor V_{\lambda_2}\tensor\cdots \tensor V_{\lambda_n}=\bigoplus \op{Hom}_{\frg}(V_{\nu}, V_{\lambda_1}\tensor V_{\lambda_2}\tensor\cdots \tensor V_{\lambda_n})\tensor V_{\nu},$$
 where the sum is over all dominant integral weights $\nu$ of $\frg$, so that the KZ connection induces a connection on each
$\op{Hom}_{\frg}(V_{\nu}, V_{\lambda_1}\tensor
V_{\lambda_2}\tensor\cdots \tensor V_{\lambda_n})$.

For describing the KZ connection on
$V_{\lambda_1}\tensor V_{\lambda_2}\tensor\cdots\tensor
V_{\lambda_n}$, it suffices therefore to describe the connection on
$\op{Hom}_{\frg}(V_{\nu}, V_{\lambda_1}\tensor V_{\lambda_2}\tensor
\cdots \tensor V_{\lambda_n})$. This is equivalent to describing the
connection on the spaces
$$(V_{\lambda_1}\tensor V_{\lambda_2}\tensor\cdots\tensor V_{\lambda_n}\tensor
V^*_{\nu})_{\frg}= (V_{\lambda_1}\tensor V_{\lambda_2}\tensor\cdots\tensor V_{\lambda_n}\tensor
V^*_{\nu})/\frg(V_{\lambda_1}\tensor V_{\lambda_2}\tensor\cdots
\tensor V_{\lambda_n}\tensor
V^*_{\nu}).$$
We will think of these spaces as being attached to
the representations $V_{\lambda_1},\dots,V_{\lambda_n}$ at
$z_1,\dots,z_n$ respectively, and the representation $V_{\nu}^*$ at
$\infty\in\mathbb{P}^1$.

There is an isomorphism (which tensors with a lowest weight vector for
$V_{\nu}^*$), \cite[Lemma 2.7.1]{FeSV2}
$$(V(\vv{\lambda})_{\mathfrak{n}^-})_{\nu}\leto{\sim} (V_{\lambda_1}\tensor V_{\lambda_2}\tensor\cdots \tensor V_{\lambda_n}\tensor V^*_{\nu})_{\frg}.$$

When $\kappa=\ell + h^{\vee}$, $\ell\in\mathbb{Z}_{>0}$, and
  $\lambda_1,\dots,\lambda_n$ and $\nu$ are of level $\ell$, the
connection on
$(V_{\lambda_1}\tensor V_{\lambda_2}\tensor\cdots\tensor
V_{\lambda_n}\tensor V^*_{\nu})_{\frg}$ induces the KZ connection on
the vector bundles of conformal blocks
$\mathbb{V}_{n+1,\vv{\lambda},\nu^*,\ell}$ (with one insertion at
infinity).

Some of the basic questions concerning the  KZ connection are:
\begin{enumerate}
\item When is the image of the monodromy representation
  finite/infinite?
\item When is the monodromy representation irreducible?
\item When does the connection preserve a unitary metric?
\end{enumerate}
In the next section we will see that motivic properties of the KZ
connection allow one to address some of these questions in certain
cases. However, even some basic local properties of the KZ connection
are unknown:
\begin{question}
  Can one describe the conjugacy classes of the local monodromy of the
  KZ connection around the boundary of $\mathcal{C}_n$ (for example, in the
  Fulton--MacPherson compactification)?
\end{question}
The answer is known for (suitably) general values of $\kappa$ as well
as for the quotient of conformal blocks when $\kappa= \ell+h^{\vee}$,
$\ell\in\mathbb{Z}_{>0}$. In general the semi-simplifications of the
conjugacy classes are known. See e.g., \cite[Section 6.5, Proposition
7.8.6]{BaKi} and \cite[Theorem 6.2.2]{TUY}.

\subsection{Motivic realization of the KZ connection}
We consider  the KZ connection on the constant bundle on $\mathcal{C}_n$ with fibre (one thinks of $V^*_{\nu}$ as placed at infinity)
$$ \mathbb{A}(\vv{\lambda},\nu^*)=(V_{\lambda_1}\tensor V_{\lambda_2}\tensor\cdots \tensor V_{\lambda_n}\tensor V^*_{\nu})_{\frg}$$
and follow \cite{L2} (completed in \cite{BBM}) to construct the
motivic realization. Note that both \cite{L2} and \cite{BBM} restrict
to the case $\nu=0$; the proofs for arbitrary dominant $\nu$ will be
given in \cite{BFM}.

We will assume that $\sum\lambda_i-\nu$ is a sum of positive simple
roots since otherwise $\mathbb{A}(\vv{\lambda},\nu^*)$ is zero.  Write
$$\sum\lambda_i-\nu=\sum_{p=1}^r n_p \alpha_p,\ \ n_p\in\mathbb{Z}_{\geq 0}, \ \mbox{and} \ \alpha_p \in \Delta.$$
  Fix a  map $\beta:[M]=\{1,\dots,M\} \to \Delta$, so that
  $ \sum\lambda_i-\nu=\sum_{b=1}^M \beta(b)$ with $M=\sum_{p=1}^r n_p.$

\subsubsection{}
Fix a point $\vv{z}=(z_1,\dots,z_n)\in\mathcal{C}_n$. Let $W=\mathbb{C}^M$. The
coordinate variables of $W$ will be denoted by $t_1,\dots,t_M$. We
will consider the variable $t_b$ to be coloured by the simple root
$\beta(b)$. Let $\kappa$ be an arbitrary nonzero complex number and
consider the weighted hyperplane arrangement $(W,\mathcal{C},a)$ in
$W$ given by the following collection of hyperplanes, and their
attached weights (see \cite{BBM} for the notation):
\begin{enumerate}
\item For $i\in [1,n]$ and $b\in [1,M]$, the hyperplane $t_b-z_i=0$, with weight $\displaystyle \frac{(\lambda_i,\beta(b))}{\kappa}$.
\item $b,c\in [1,M]$ with $b<c$,  the hyperplane $t_b-t_c=0$, with weight $\displaystyle-\frac{(\beta(b),\beta(c))}{\kappa}$.
\end{enumerate}
Let $U$ be the complement of the above hyperplane arrangement in $W$.  Let
\begin{equation}\label{deta}
\eta_a=\frac{1}{\kappa}\biggl(\sum_{i=1}^n\sum_{b=1}^M
(\lambda_i,\beta(b))\frac{d(t_b-z_i)}{t_b-z_i} - \sum_{1\leq b<c \leq M} (\beta(b),\beta(c))\frac{d(t_b-t_c)}{t_b-t_c}\biggr) .
\end{equation}
\begin{defi}
The Aomoto subalgebra   $A^{\starr}(U)$ of $\Omega^{\starr}(U)$ generated over $\mathbb{C}$ by the closed forms $\dlog(t_b-z_i)$ and $\dlog(t_b-t_c)$ corresponding to the hyperplanes in the arrangement.  It is equipped with a differential $d\omega:=\eta_a\wedge \omega$. We call the pair $(A^{\starr}(U), \eta_a \wedge)$, the Aomoto complex.
 Let $\mathcal{L}(a)$ be the rank one local system on $U$ corresponding to the one-form $\eta_a$ (the kernel of $d+\eta_a$).

 Let $\Sigma_M$ be the subgroup of the symmetric group $S_M$ on $[M]=\{1,\dots,M\}$ formed by permutations $\sigma$ that preserve colour, i.e. $\beta(\sigma(a))=\beta(a)$ for all $a\in[M]$;
 $\Sigma_M$ acts on $U$, the complex $(A^{\starr}(U), \eta_a \wedge)$ as well as the local system $\ml(a)$ (over the action on $U$).

\end{defi}
Let $P$ be any $\Sigma_M$ equivariant smooth projective
compactification of $U$, with $P-U=\bigcup_{\alpha} E_{\alpha}$ a
divisor with normal crossings.  Let $V=P-\bigcup'_{\alpha}E_{\alpha}$,
where the union is restricted to $\alpha$ such that
$a_{\alpha}=\op{Res}_{E_{\alpha}}\eta_a$, the residue of $\eta$, is
not a strictly positive integer {and let $j:U \to V$ be the inclusion.}
\begin{lemma}\label{topo}%
  \cite[Proposition 4.2]{L1} There is a $\Sigma_M$-equivariant
  isomorphism
  $H^*(A^{\starr}(U), \eta_a
  \wedge)\leto{\sim}H^{\starr}(V,j_{!}\mathcal{L}(a))$.
\end{lemma}
Note that the spaces $H^*(A^{\starr}(U), \eta_a \wedge)$ remain the same as $\eta_a$ is scaled, while their topological description changes.

\subsubsection{}
There is a map constructed by Schechtman and Varchenko \cite{SV}
\begin{equation}\label{SV_map}
V(\vv{\lambda})^*_{-\nu} \to A^M(U)^{\chi}
\end{equation}
where $\chi$ is the sign character on $S_M$ restricted to $\Sigma_M$.
Note that  $\mathbb{A}(\vv{\lambda},\nu^*)^*$ is a subspace of $(V(\vv{\lambda})^*_{-\nu})$.
\begin{remark}\label{unitar}
  Suppose $\kappa=\ell+h^{\vee}$ with $\ell\in \mathbb{Z}_{>0}$, and
  $\widehat{U}\to U$ be the cover corresponding to the
  Schechtman--Varchenko \emph{master function}, i.e., a multivalued
  function whose differential is $\eta_a$. Note that there is a
  natural map $A^M(U)\to \Omega^M(\widehat{U})$ given by
  multiplication with the master function.

Dual conformal blocks at level $\ell$ on $\mathbb{P}^1$ with insertions $\lambda_i$ at $z_i$ and $\nu^*$ at $\infty$ inject into $\mathbb{A}(\vv{\lambda},\nu^*)^*$ and hence into
 $(V(\vv{\lambda})^*_{-\nu})$. The image of \eqref{SV_map} on elements from dual conformal blocks give us top degree differential forms on $\widehat{U}$.
 A key property, conjectured  by Gawedzki and his collaborators
 \cite{Gaw3,Gaw2}, and proved by Ramadas \cite{TRR}
 for $\sl_2$ and in \cite{b-unitarity} in general for $\nu=0$, is that
 these forms extend to any smooth projective compactification of
 $\widehat{U}$. It can be checked using the method of proof of
 \cite{b-unitarity} that this is true for all $\nu$. We sketch the
 changes that need to be made below.


 Briefly, in the proof of \cite[Theorem 3.5]{b-unitarity} we have to
 work with the insertion of the lowest weight vector of $V_{\nu}^*$ at
 $\infty$ in the correlation function $\Omega$ which has weight
 $-\nu$. We first observe that the correlation functions remain
 holomorphic as $t_i=\infty$. We then only need to revisit the
 calculations as a bunch of the $t$-variables approach infinity, i.e.,
 the (S3) strata. The formula (6.4) in \cite{b-unitarity} changes to
 $(\gamma,\gamma)$ replaced by $2(\nu,\gamma)+(\gamma,\gamma)$. The
 bound from representation theory from \cite[Lemma 4.5]{b-unitarity}
 is now $(2(\nu,\gamma)+(\gamma,\gamma))/2k$ instead of
 $(\gamma,\gamma)/2k$. The rest of the argument just replaces
 $(\gamma,\gamma)$ by $2(\nu,\gamma)+(\gamma,\gamma)$.
 \end{remark}

\begin{clm} \label{inj}
The composite $$\mathbb{A}(\vv{\lambda},\nu^*)^*\to V(\vv{\lambda})^*_{-\nu} \to A^M(U)^{\chi}\to H^M(A^{\starr}(U), \eta_a \wedge)^{\chi}$$
is injective (for arbitary $\kappa\neq 0$).
\end{clm}

This property is proved in \cite[Proposition 1.5]{BBM} for $\nu=0$. The argument holds for all $\nu$: For sufficiently large
integers $\ell>0$, and $\kappa=\ell+h^{\vee}$,
$\mathbb{A}(\vv{\lambda},\nu^*)^*$ coincides with the space of conformal
blocks (with insertions of $\lambda_i$ at $z_i$, and $\nu^*$ at
$\infty$) and for conformal blocks, the composite to
$H^M(U,\mathcal{L}(a))$ is injective by \cite{TRR,b-unitarity} (see
Remark \ref{unitar}) and results of Hodge theory (top degree
holomorphic differential forms on a smooth projective variety inject
into the cohomology of any non-empty open subset).  Since the groups
$H^M(A^{\starr}(U), \eta_a \wedge)$ are insensitive to scaling of
$\eta$, the desired injection follows for all $\kappa\neq 0$.

\subsubsection{}\label{parallel}

Using Lemma \ref{topo} we get an injective map
$\mathbb{A}(\vv{\lambda},\nu^*)^*\hookrightarrow
H^{M}(V,j_{!}\mathcal{L}(a))^{\chi}$ which is known to preserve connections
(KZ on one side and Gauss--Manin on the other, see \cite[Theorem
3.6]{L2}). To determine the image, Looijenga \cite{L2} proceeds as
follows: he considers the same map for $-\kappa$ and dual weights. He
then dualizes this map and gets a surjection
$$H^M(V',q_{!}\mathcal{L}(a))^{\chi}\twoheadrightarrow \mathbb{A}(\vv{\lambda}^*,\nu)$$
where $V'=P-\bigcup'_{\alpha}E_{\alpha}$, the union is restricted to
$\alpha$ such that $a_{\alpha}=\op{Res}_{E_{\alpha}}\eta$, is not a
non-negative integer. Note that $V'\supseteq V$, and $q:U\to
V'$. Using the identity
$\mathbb{A}(\vv{\lambda}^*,\nu)=\mathbb{A}(\vv{\lambda},\nu^*)^*$,
Looijenga realizes the image of
$\mathbb{A}(\vv{\lambda},\nu^*)^*\hookrightarrow
H^{M}(V,j_{!}\mathcal{L}(a))^{\chi}$ as the image of the map
$$H^M(V',q_{!}\mathcal{L}(a))^{\chi}\to H^M(V,j_{!}\mathcal{L}(a))^{\chi}.$$
That this  map defined here using representation theory coincides with the natural map from topology is
proved in \cite[Theorem 7.3]{BBM} for $\nu=0$. The argument (which uses a formula
for the topological map proved in \cite{BBM}, as well as results from \cite{SV}) extends to
all $\nu$. All in all, one has the following topological description
of invariants:
\begin{proposition} \label{p:kz}
  $\mathbb{A}(\vv{\lambda},\nu^*)^*$ equals (compatibly with the KZ
  connection) the image of the map
  $H^M(V',q_{!}\mathcal{L}(a))^{\chi}\to H^{M}(V,j_{!}\mathcal{L}(a))^{\chi}$.
\end{proposition}

Denote the image of $H^M(V',q_{!}\mathcal{L}(a))^{\chi}\to
H^{M}(V,j_{!}\mathcal{L}(a))^{\chi}$ by
$\kz(\vv{\lambda},\nu^*,\vv{z})$,
which is the stalk at $\vv{z}$ of a local system  $\kz=\kz(\vv{\lambda},\nu^*)$ on $\mathcal{C}_n$. The corresponding vector bundle is the trivial bundle with fibres $\mathbb{A}(\vv{\lambda},\nu^*)^*$.

It is noted in \cite{L2} that the above proposition implies that
$\kz(\vv{\lambda},\nu^*,\vv{z})$ carries a mixed Hodge
structure with coefficients in a cyclotomic subfield
$k$\footnote{\label{f:coef} This means, see \cite[\S 2.1, B]{del-val},
  that it is an object of the Karoubian envelope of the ``base
  change'' to $k$ of the $\Q$-linear abelian category of $\Q$-mixed
  Hodge structures. In this article, an MHS, VMHS or a motive is
  always allowed to have coefficients in a cyclotomic field (which we
  do not always specify).}  of $\mathbb{C}$ (depending on $\kappa$,
$\vl$ and $\nu$) if $\kappa$ is a nonzero rational number.  The
varieties, sheaves and maps occurring in Proposition \ref{p:kz} may be
spread out over $\mathcal{C}_n$, the key ingredient is the
(multivalued) master function
\begin{equation} \label{e:master} %
  \mathcal R = \prod_{1 \leq i < j \leq n} (z_i -
  z_j)^{\frac{(\lambda_i, \lambda_j)}{\kappa}} \prod_{b=1}^M
  \prod_{j=1}^n (t_b - z_j)^{\frac{(\lambda_j, \beta(b))}{\kappa}}
  \prod_{1 \leq b < c \leq M} (t_b - t_b)^{-\frac{(\beta(b),
      \beta(c))}{\kappa}}
\end{equation}
where now the $z_i$'s are variables.  Let $C$ be the minimal positive
integer such that $C\kappa$ multiplied by all the exponents occurring
in \eqref{e:master} are integers.  Then by the theory of mixed Hodge
modules, one gets an admissible variation of mixed Hodge structure
(VMHS) on $\mc_n$ (with coefficients in
$k = \Q(\mu_{C\kappa}) \subset \C$) with fibre over any point
$\vec{z}$ being the MHS defined as the image of the map in Proposition
\ref{p:kz}.

In \cite{morel-ivorra} the authors define a category $\mathscr{M}(X)$
of perverse mixed motivic sheaves on any variety defined over a
subfield of $\C$ and construct functors on
$\mathrm{D}^b(\mathscr{M}(X))$ which are analogs of the usual
operations on the derived category of constructible sheaves. In
particular, they construct $f_*$ and $f_!$ functors from
$\mathrm{D}^b(\mathscr{M}(X))$ to $\mathrm{D}^b(\mathscr{M}(Y))$ for
any map of varieties $f:X \to Y$ defined over $k$. It follows that the
VMHS defined above can be upgraded to an object of
$D^b(\mathscr{M}(\mc_n))$ defined over $k$ and with coefficients also
in $k$. The Betti realisation of this object is a local system, so
this object can be viewed as a motivic local system.  The stalk of
this motivic local system at each point $\vv{z} \in \mc_n$ gives a
Nori motive (see, e.g., \cite{huber-ms}) defined over the residue
field $k(\vv{z})$ of the point, where now we view $\mathcal{C}_n$ as
being defined over $k$. We will refer to these motives as \emph{KZ
  motives}.  We note that these motives, therefore also the
corresponding mixed Hodge structures, are independent\footnote{See,
  for example, \S 2.6 and \S 2.7 of the arXiv version
  ({https://arxiv.org/pdf/1611.01861.pdf}) of \cite{BBM}.} of the
choice of compactification $P$ of $U$.

\begin{question} \label{q:hodge} %
  Can one compute the ranks of the Hodge and weight filtrations of KZ
  motives, and their behavior under factorization?
\end{question}
Such computations would have important consequences for the monodromy
of the KZ system. For example, if the MHS is pure then the monodromy
must be semi-simple and if the MHS is mixed Tate then the monodromy
representation has a filtration such that the action of
$\pi_1(\mathcal{C}_n)$ on the associated graded vector space factors
through a finite quotient. However, even in the cases $n=2,3$, when the
monodromy is (projectively) trivial, we do not know any simple method
for computing the numerical invariants of the MHS.

  \begin{example} \label{e:legendre} Suppose $\frg = \sl_2$, $n=4$ and
    each $\l_i$ is the fundamental weight (and $\nu = 0$). The rank of
    the coinvariants in this case is $2$.
  \begin{enumerate}
  \item $\kappa = 1$: The KZ motives are Tate motives of
    weight $0$.
  \item $\kappa = 2$: The KZ motivic local system is a twist of the
    local system associated to the $H^1$ of the family of double
    covers of $\mathbb{P}^1$ ramified along the $4$ marked points of
    $\mathbb{A}^1$ corresponding to each $\vv{z}$.
  \item $\kappa = 3$: The KZ motives in this case are mixed with
    nonzero Hodge numbers $h^{2,0} = h^{2,1} = 1$. The coefficient
    field is $\Q(\mu_3)$ which is an imaginary quadratic extension of
    $\Q$, so by taking the sum of these motives with their complex
    conjugates one obtains motives with $\Q$-coefficients. In each
    weight equal to $2$ or $3$ one gets a rank $2$ motive with CM by
    $\Q(\mu_3)$, so the Hodge structure arises from any elliptic curve
    $E$ over $\C$ with CM by this field as follows: Let $M$ be
    $H^{1,0}(E)$, viewed as an MHS with coefficients in
    $\Q(\mu_3)$. Then the weight $2$ Hodge structure must be
    $M^{\otimes 2}$ and the weight $3$ Hodge structure must be
    $M(-1)$. In particular, these pure subquotients do not depend on
    the point $\vv{z}$, but the extension does depend on the point
    (see \cite[\S 9]{BBM}). One can make a guess as to the precise
    nature of this extension, but this seems difficult to verify.
  \item $\kappa > 3$: In this case the coinvariants are equal to the
    conformal blocks. It follows from \cite{TRR} (see also Proposition
    \ref{p:pure} below) that the KZ motives are pure and the only
    nonzero Hodge number is $h^{2,0}$.
  \end{enumerate}
\end{example}
 Further examples of explicit computations of KZ motives will appear in
\cite{BFM}.

\begin{remark} \label{r:conj}%
  It would also be interesting to compute the
  $\mathrm{Gal}(k/\Q)$-conjugates of a KZ motives, in particular their
  Hodge numbers, where the Galois group acts on the coefficient
  field. It is not clear how to do this in general, even in the case
  when coinvariants are equal to conformal blocks.
\end{remark}

\begin{lemma} \label{l:finite}%
  Let $M$ be a (geometric) VHS on an algebraic variety with
  coefficients in a number field $k$. If all the
  $\mathrm{Gal}(k/\Q)$-conjugates of $M$ have no non-vanishing
  adjacent Hodge numbers, i.e., no two numbers of the form $h^{p,q}$
  and $h^{p-1,q+1}$ are simultaneously nonzero, then the monodromy of
  $M$ is finite. If the monodromy of $M$ is irreducible, then it is
  finite if and only if every $\mathrm{Gal}(k/\Q)$-conjugate of $M$
  has a unique non-vanishing Hodge number.
\end{lemma}

\begin{proof}
  Let $M'$ be the sum of the $\mathrm{Gal}(k/\Q)$-conjugates of $M$,
  so $M'$ is a $\Q$-VHS and can be given a $\Z$-structure (by
  geometricity). Now the fact that each conjugate has no non-vanishing
  adjacent Hodge numbers (and is preserved by the connection on $M'$)
  implies by Griffiths transversality and polarisability that the
  local system underlying $M'$ has a unitary structure. Since the
  $\Z$-structure implies that the monodromy of $M'$ is discrete, it
  follows that the monodromy of $M'$, therefore also $M$, must be
  finite.

  If the monodromy is irreducible, then by Griffiths transversality
  the non-zero Hodge numbers of any $\mathrm{Gal}(k/\Q)$-conjugate
  must be of the form $h^{p,q}, h^{p-1,q+1},\dots, h^{p-r,q+r}$, for
  some $r \geq 0$. If the monodromy is also finite then $r$ must be
  $0$ since a non-degenerate invariant Hermitian form on an
  irreducible representation of a finite group must be definite.

\end{proof}

\begin{remark} %
  The use of the Galois action to prove finiteness of monodromy of KZ
  equations has also been used in the physics literature (see, e.g.,
  \cite{ST}).
\end{remark}

\begin{example} \label{e:fin} %
  We consider the case (4) of Example \ref{e:legendre} with
  $\kappa = 4, 6$. In both these cases the field $k$ can be taken to
  be an imaginary quadratic field. Since the only nonzero Hodge number
  of the VHS is $h^{2,0}$, the only nonzero Hodge number of its
  complex conjugate is $h^{0,2}$, so by Lemma \ref{l:finite} we
  recover the known fact that the monodromy is finite in these
  cases. The only other case in which the monodromy is finite is
  $\kappa = 10$ \cite{M, LPS})\footnote{This is also proved in
    \cite[Theorem 3.3]{ST} under the assumption that the monodromy is
    irreducible.}; finiteness in this case can also be proved by
  analysing the action of $\mathrm{Gal}(k/\Q)$, but the argument is
  more elaborate.

\end{example}

\begin{question} \label{q:fin}
  Is there an algorithm to determine when the monodromy of
  $\kz(\vv{\lambda},\nu^*)$ is finite?
\end{question}
We ask for an algorithm instead of a classification since the latter
seems like a much more difficult question. Aside from Hodge theoretic
methods as above, one might also try to use more arithmetic methods:
The relevance of Grothendieck's $p$-curvature conjecture to this
question (in the case of conformal blocks) is pointed out in
\cite[Appendix A]{andre}. We also note that the results of
\cite{katz-as} and Propositions \ref{p:kz} and \ref{p:pure} imply that
Grothendieck's conjecture is true for any pure subquotient of
$\kz(\vv{\lambda},\nu^*)$ and also for conformal blocks (for
which see also \cite[Corollary 19.9] {BFS}.

\subsubsection{}
The following is an elaboration of \cite[Question 7.6]{BBM}.
\begin{example}
  Let $\lambda_1,\dots,\lambda_n$ be highest weights for $\sl_2$ with
  $\sum \lambda_i=2m\in 2\mathbb{Z}$, $0<\lambda_i<m$. Consider the
  local system $A(\vv{\lambda})^*$ on $\mathcal{C}_n$ at level
  $\ell=m-2$, so that $\kappa=m$. It has a sublocal system given by
  dual conformal blocks $\mathbb{V}^*$.  The quotient
  $A(\vv{\lambda})^*/\mathbb{V}^*$ on $\mathbb{C}^n$ is essentially
  the Lauricella system considered in \cite{DMo} for the weights
  $\mu_i=\lambda_i/m$: The number of $t$ variables is $m$. The stratum
  when $\ell+1$ of the $t$ points are at infinity detects conformal
  blocks. Therefore we can take residues and the
  $A(\vv{\lambda})^*/\mathbb{V}^*$ maps to the cohomology of an open
  curve, which is the one corresponding to the Lauricella system (more
  details will appear in \cite{BFM}). {We note that the Hodge numbers
    of the associated motives are computed in \cite[Proposition
    2.20]{DMo}.}
\end{example}

\subsection{Motives associated to conformal blocks}

\subsubsection{}
Let $M$ be a pure Nori motive of weight $w$ over $k$ with
$\Q$-coefficients.
\begin{defi} \label{d:pol}%
  A polarisation on $M$ is a morphism of motives
  $\langle\, , \rangle: M \otimes M \to \Q(-w)$ such that each Hodge
  realisation is a polarisation of the corresponding $\Q$-Hodge
  structure (\cite[D\'efinition 2.1.15]{hodge2}).
\end{defi}
We analogously define a polarisation of a pure motivic local
system.

A polarised $\Q$-motive is a $\Q$-motive together with a
polarisation. A polarised (pure) motive with coefficients in $L$, where $L$ is
any extension of $\Q$, is a direct summand of $M \otimes_{\Q} L$ (see
footnote \ref{f:coef}) where $M$ is a polarised $\Q$-motive (and
similarly for motivic local systems).

\subsubsection{}
 Suppose $\lambda_1,\dots,\lambda_n$ and $\nu$ are dominant
  integral weights of $\frg$ of level $\ell$.  Let $\cblock$ be the
fibre of the conformal block bundle
$\mathbb{V}_{n+1,\vv{\lambda},\nu^*,\ell}$ at $\vec{z}$ (i.e., with an
insertion of $\nu^*$ at $\infty\in\Bbb{P}^1$); it is a quotient of
$\mathbb{A}(\vv{\lambda},\nu^*)$.
\begin{proposition} \label{p:pure}%
  Suppose $\kappa=\ell+h^{\vee}, \ell\in \mathbb{Z}_{>0}$. Then, the
  space of dual conformal blocks $\cblock^*$ coincides with the
  subquotient of the image of
  $H^M(V',q_{!}\mathcal{L}(a))^{\chi}\to
  H^M(V,j_{!}\mathcal{L}(a))^{\chi}$ given by the image of
  \begin{equation}\label{overflow2}
  H^M_c(U,\mathcal{L}(a))^{\chi}\to H^{M}(U,\mathcal{L}(a))^{\chi}.
  \end{equation}
  Therefore, dual conformal blocks $\cblock^*$ give a canonical
  (polarised) motivic local system such that the Hodge realisation for
  the inclusion $k \hookrightarrow \mathbb{C}$ is a (polarised) VHS on
  $\mathcal{C}_n$ concentrated in degree $(M,0)$.
\end{proposition}
\begin{proof}
There is a commutative diagram:
\begin{equation}\label{commie}
\xymatrix{
\mathbb{A}(\vv{\lambda}^*,\nu)\ar[r] & \mathbb{A}(\vv{\lambda},\nu^*)^*\ar[d]\\
H^M(V',q_{!}\mathcal{L}(a))^{\chi}\ar[r]\ar[u] & H^{M}(V,j_{!}\mathcal{L}(a))^{\chi}\ar[d]\\
H^M_c(U,\ml(a))^{\chi}\ar[u]\ar[r] & H^M(U,\ml(a))^{\chi}.
}
\end{equation}
Since correlation functions for dual conformal blocks $\cblock^*$ (the
image under \eqref{SV_map} of
$\cblock^*\subseteq \mathbb{A}(\vv{\lambda},\nu^*)^*\subseteq
V(\vv{\lambda})^*_{-\nu}$) extend to smooth projective
compactifications of $\widehat{U}$ \cite{b-unitarity} (see Remark
\ref{unitar}), $\cblock^*$ is naturally a subspace of the image
$H^M_c(U,\mathcal{L}(a))^{\chi}\to
H^{M}(U,\mathcal{L}(a))^{\chi}$. (Here we use the fact, which follows
from the long exact sequence of cohomology with compact support in the
category of mixed Hodge structures, that global top degree holomorphic
forms on a smooth proper variety lift to compactly supported
cohomology classes on Zariski open subsets.  Also, top degree
holomorphic differential forms on a smooth projective variety inject
into the cohomology of any non-empty open subset.  This gives that
$\cblock^*$ injects into $H^{M}(U,\mathcal{L}(a))^{\chi}$.)

We only need to show that this image is not any bigger. Using
\eqref{commie}, it suffices to prove that the composite,
\begin{equation}\label{overflow}
H^M_c(U,\mathcal{L}(a))^{\chi}\to \mathbb{A}(\vv{\lambda},\nu^*)^*
\end{equation}
lands inside dual conformal blocks $\cblock^*$.

By duality, this is the same as asking that the map
$\mathbb{A}(\vv{\lambda},\nu^*)\to H^M(U,\ml(-a))^{\chi}$ factors
through the space of conformal blocks $\cblock$. This is proved in
\cite[Theorem 4.3.1]{FeSV2} by showing that elements in the kernel of
the map from $\mathbb{A}(\vv{\lambda},\nu^*)$ to conformal blocks give
rise to exact forms on $U$ 
(i.e., $d\omega$ where $\omega$ is the master function times a form
with (possibly) non-log poles, see the example on page 287 in
\cite{VBook}.)

A canonical polarisation on the image of
$H^M_c(U,\mathcal{L}(a))^{\chi}\to H^{M}(U,\mathcal{L}(a))^{\chi}$ is
inherited from the cohomology of a smooth projective compactifcation
of $\widehat{U}$. This polarisation is independent of the
compactification because conformal blocks are of type $(M,0)$ on the
compactification, and the corresponding polarisation comes from
integrals over $\widehat{U}$.
\end{proof}

Proposition \ref{p:pure} answers \cite[Question 6.2]{LV} in the
affirmative. It also shows that \cite[Theorem 1.1]{BMu} holds without
restrictions on the simple Lie algebra $\frg$ by an essentially
different argument: the space of all suitably invariant top degree
holomorphic differential forms on any smooth compactification $Y$ of
$\widehat{U}$ (as in Remark \ref{unitar}) is naturally isomorphic to
the space of dual conformal blocks $\cblock^*$.

\begin{corollary}
 The  image of \eqref{overflow} is exactly the space of dual conformal blocks $\cblock^*$. Dualising, we obtain an injective map $\cblock\to H^M(U,\ml(-a))^{\chi}$. The image of this map is the image of
 $H^M_c(U,\ml(-a))^{\chi}\to H^M(U,\ml(-a))^{\chi}$.
\end{corollary}
This affirmatively answers  the question following the theorem on page 220 of
\cite{FeSV2}.
\begin{proof}
 Proposition \ref{p:pure} gives  $H^M_c(U,\ml(a))\twoheadrightarrow \cblock^* \hookrightarrow H^M(U,\ml(a))$, we dualise to get
$$H^M_c(U,\ml(-a))\twoheadrightarrow \cblock \hookrightarrow H^M(U,\ml(-a)),$$
 which shows $\cblock \to H^M(U,\ml(-a))$ is injective, and also identifies its image.

\end{proof}

\subsubsection{}

It is shown in \cite[Theorem 7.2.2] {BA} that (irreducible) unitary rigid
local systems (in the sense of \cite{Katz}) on $\mathbb{P}^1-S$, $S$
a finite subset, with finite local monodromies  arise
essentially as solutions of KZ equations on conformal blocks for
special linear groups (one fixes all but one of the $z_i$).  There are
examples of unitary rigid (irreducible) local systems with infinite
monodromy (see, e.g., \cite{BH}), and therefore the corresponding KZ
connections on conformal blocks have infinite monodromy. We note that
examples of TQFT representations, which are equivalent to the
monodromy of WZW/KZ equations on conformal blocks, with infinite monodromy
were first given in \cite{M}.

\begin{question}
  Do all rigid local systems on $\mathbb{P}^1-S$, $S$ a finite subset,
  with quasi-unipotent local monodromies arise as subquotients of KZ
  equations on spaces of coinvariants (allowing negative, or rational
  values for $\kappa$)?
\end{question}

\begin{remark} \label{r:red}%
  The monodromy representations associated to conformal blocks need
  not always be irreducible and it would be very interesting to have a
  criterion for irreducibility. Several reducible examples are given in
  \cite[\S 3.3]{ST}, in particular, it is shown that the monodromy
  representation on conformal blocks for $\sl_2$ with $n= 4$,
  $\ell = 10$ and all $\lambda _i = \lambda$, is reducible when
  $\lambda = 4\varpi$ or $6 \varpi$ (with $\varpi$ the fundamental
  weight).
\end{remark}

\subsubsection{}

Another question relates to strange duality isomorphisms. These relate
conformal blocks for one $\frg$ with another simple Lie algebra
$\frg'$ (see e.g., Mukhopadhyay's appendix in \cite{KBook} for the
background) with these isomorphisms being flat for the KZ
connection. Since we have motivic realizations for conformal blocks,
we may ask

\begin{question}\label{subtle}
  Are the isomorphisms arising from strange dualities isomorphisms of
  polarised motives?

  Roughly speaking we are asking if the strange duality maps arise
  from algebro-geometric maps or correspondences.  In particular, this
  would imply that they preserve the unitary metrics coming from
  Proposition \ref{p:pure}.  This compatibility of metrics was also
  raised as a question by Minwalla to the first named author.
\end{question}

Note that strange dualities arise from conformal embeddings
$\frg_1\to\frg_2$ (e.g., for the well studied type A case the
embedding is the tensor product $\sl_m\oplus\sl_n\to \sl_{mn}$). One
then restricts suitably to get dual conformal blocks for $\frg_2$ at
level one (which are one dimensional) to map to dual conformal blocks
for $\frg_1$, with the map projectively flat for KZ connections. We
are asking if this map is motivic. In the type A case considered above
conformal blocks for $\frg_1$ break up as a tensor product of $\sl_m$
blocks at level $n$ and $\sl_n$ blocks at level $m$ and the one
dimensional $\frg_2$ conformal blocks give a perfect duality
(therefore the duality involves tensoring with the one dimensional
motive for the $\frg_2$ blocks).

Question \ref{subtle} is subtle because the number of variables $M$
for $\frg_1$ and $\frg_2$ are in general different for a conformal
embedding. { Also, since the $\kappa$ corresponding to
  $\mathfrak{g}_1$ and $\mathfrak{g}_2$ is in general not the same,
  one might need to enlarge the coefficient field as well as the base
  field in order for the strange duality isomorphisms to be defined
  motivically.}

\begin{example}\label{newexample}
Consider a 2 dimensional complex vector space $V$, and the map of groups $\SL(V)\times \SL(V)\to \SL(V\tensor V)$, i.e., a map $\SL_2\times \SL_2\to \SL_4$. The standard representation $W$ of $\SL_4$ restricts to the tensor product of standard representations. Therefore we have a map
$$(V^{\tensor 4})_{\sl_2}\tensor (V^{\tensor 4})_{\sl_2}\to (W^{\tensor 4})_{\sl_4}$$
which gives a map of constant vector bundles on $\mathcal{C}_4$. The
coinvariants are conformal blocks, at level $1$ for $\SL_4$ and level
$2$ for the $\sl_2$ factors.
Dualizing, we get  a map $$(W^{\tensor 4})^*_{\sl_4}\to (V^{\tensor 4})^*_{\sl_2}\tensor (V^{\tensor 4})^*_{\sl_2}.$$

This is an example of strange duality: The rank of
$(W^{\tensor 4})_{\sl_4}$ is one and the map is projectively flat for
connections. The (central) monodromy 
as any of the $4$ points go around infinity can be computed: the
exponents are $2c(\omega_1)/(2+2)$ on the $\sl_2$ side which equals
$c(\omega_1)/5$ on the $\sl_4$ side, so perhaps the map is actually
flat.

For the motive associated to the $\sl_2$ conformal blocks we have two
$t$ points each, so the tensor product gives us four points. For
$\sl_4$ we have $4L_1=3\alpha_1+2\alpha_2+\alpha_3$, so $6$
points. Naively, we have the type $(6,0)$ cohomology of a
$6$-dimensional variety mapping to the type $(4,0)$ cohomology of a
$4$-dimensional one. Is there a motivic mechanism for such a map?
\end{example}

\subsubsection{}

The center $Z(G)$ of the simply connected form $G$ of $\frg$ acts on
the set of level $\ell$ weights of $\frg$ \cite{FS}. Denote this
action by $\lambda\mapsto c\lambda$ with $c\in Z(G)$. Suppose we have
elements $c_1,\dots,c_n\in Z(G)$ with $c_1c_2\dots c_n=1$, then there
is an isomorphism (\cite{FS}, see also \cite{SM}) of conformal blocks
bundles (preserving projective KZ connections) at level $\ell$ on
$\mathcal{C}_n$ with insertions $(\lambda_1,\dots,\lambda_n;0)$ ($0$
insertion at infinity), and insertions
$(c_1\lambda_1,\dots,c_n\lambda_n;0)$.

\begin{question}
  Are the isomorphisms of conformal blocks bundles arising from the
  action of the center of $G$ isomorphisms of polarised motives, perhaps after
  twisting by a rank one motive?

  (We need to allow twisting by rank one motives since the number of
  $t$ variables, therefore the Hodge numbers, can be different.)

\end{question}

\subsubsection{}

In higher genus, there is a projective WZW/Hitchin connection on the
bundle of conformal blocks $\mathbb{V}_{g,n,\vv{\lambda},\ell}$ on
${M}_{g,n}$ (but there is no analogue of the surjection from the
larger space of coinvariants). After pulling back
$\mathbb{V}_{g,n,\vv{\lambda},\ell}$ to a suitable cover of ${M}_{g,n}$,
one can get an actual connection and then one asks:
\begin{question}
  Is there a motivic realization of the WZW/Hitchin connection in
  higher genus as a variation of pure motives (with coefficients in a
  cyclotomic field) of type $(M,0)$ (so that one gets a unitary
  connection)?
\end{question}

In \cite[Corollary 19.9] {BFS}, using results of \cite{Arkh}, it is
shown that the local system of conformal blocks is a subquotient of a
motivic local system which is pure,  hence semisimple. One would also
 like to show the  unitarity of the local systems in a geometric way as in
 the genus zero case \cite{TRR,b-unitarity}.

\begin{remark}
  In \cite{BFS} conformal blocks are defined using the category of
  factorizable sheaves $\mathcal{F}\mathcal{S}$, which is equivalent
  to a category $\mathcal{C}$ constructed using Lusztig's small
  quantum group (see sections I.10, and V.19 of \cite{BFS}). The
  ribbon structure on the category $\mathcal{C}$ gives conformal
  blocks defined this way the structure of a local system on moduli of
  pointed curves with formal coordinates. The equivalence of this
  definition with that of conformal blocks from the WZW model follows
  from results of \cite[Section IV.9.2]{BFS}. Perhaps one can think of
  the above question as the de Rham counterpart of the Betti picture
  of \cite{BFS}.

  By \cite[Theorem 5.3]{Arkh}, conformal blocks defined using
  $\mathcal{C}$ inject into a semi-infinite Tor (an object in
  $\mathcal{C}$) which is shown to arise as a geometric local system
  \cite[Theorem IV.8.11]{BFS}. The image of the injection of conformal
  blocks into the semi-infinite Tor can perhaps be characterized in
  $\mathcal{C}$ (see \cite[I.10.1]{BFS}) as the image of a map in
  $\mathcal{C}$ but it is not immediately clear why the corresponding
  local system, and the resulting maps come from topology (somewhat in
  parallel to Section \ref{parallel}).
\end{remark}

\printbibliography


\vspace{0.05 in}

\noindent
P.B.: Department of Mathematics, University of North Carolina, Chapel
Hill, NC 27599, USA\\
{{email: belkale@email.unc.edu}}

\vspace{0.08 cm}

\noindent
N.F.: School of Mathematics, Tata Institute of Fundamental Research, Homi Bhabha Road, Mumbai 400005, India\\
{{email: naf@math.tifr.res.in}}
\vspace{0.08 cm}

\end{document}